\documentclass[preprint,12pt]{elsarticle}

\usepackage{amssymb}
\usepackage{amsmath}
\usepackage{amsthm}

\journal{Discrete Applied Mathematics}

\newtheorem{theorem}{Theorem}
\newtheorem{lemma}[theorem]{Lemma}

\newtheorem{proposition}[theorem]{Proposition}

\usepackage{xcolor,graphicx,tikz}
\usetikzlibrary[patterns]
\def\NEarrow#1{#1{\nearrow}}
\def\SEarrow#1{#1{\searrow}}
\def\NWarrow#1{#1{\nwarrow}}
\def\SWarrow#1{#1{\swarrow}}
\def\Westarrow#1{#1{\leftarrow}}
\def\Eastarrow#1{#1{\rightarrow}}
\def\Northarrow#1{#1{\uparrow}}
\def\Southarrow#1{#1{\downarrow}}
\def\NEArrow#1#2{#1{\nearrow}_{#2}}
\def\SEArrow#1#2{#1{\searrow}_{#2}}
\def\NWArrow#1#2{#1{\nwarrow}_{#2}}
\def\SWArrow#1#2{#1{\swarrow}_{#2}}
\def\WestArrow#1#2{#1{\leftarrow}_{#2}}
\def\EastArrow#1#2{#1{\rightarrow}_{#2}}
\def\NorthArrow#1#2{#1{\uparrow}_{#2}}
\def\SouthArrow#1#2{#1{\downarrow}_{#2}}
\def\North{{\rm North}}
\def\South{{\rm South}}
\def\East{{\rm East}}
\def\West{{\rm West}}

\begin{document}

\begin{frontmatter}


\title{The zero blocking numbers of grid graphs}
\cortext[corr]{Corresponding author}
\author{Hau-Yi Lin\fnref{nycu}}
\ead{zaq1bgt5cde3mju7@gmail.com}
\author{Wu-Hsiung Lin\corref{corr}\fnref{nycu}}
\ead{wuhsiunglin@nctu.edu.tw}
\author{Gerard Jennhwa Chang\fnref{ntu}}
\ead{gjchang@math.ntu.edu.tw}
\address[nycu]{Department of Applied Mathematics, National Yang Ming Chiao Tung University, Hsinchu 30010, Taiwan}
\address[ntu]{Department of Mathematics, National Taiwan University, Taipei 10617, Taiwan}






\begin{abstract}
In a zero forcing process,
vertices of a graph are colored black and white initially,
and if there exists a black vertex adjacent to exactly one white vertex,
then the white vertex is forced to be black.
A zero blocking set is an initial set of white vertices
in a zero forcing process such that ultimately there exists a white vertex.
The zero blocking number is the minimum size of a zero blocking set.
This paper gives the exact value of the zero blocking number of grid graphs.
\end{abstract}




\end{frontmatter}




\section{Introduction}\label{sec1}

In a {\em zero forcing process} on a graph $G$,
the vertices of $G$ are initially partitioned into two subsets ${\cal B}$ and ${\cal W}$,
and colored {\em black} and {\em white}, respectively.
If there exists a black vertex adjacent to exactly one white vertex,
then the white vertex is {\em forced} to be black.
If ${\cal B}$ ultimately expands to include $V(G)$, the vertex set of $G$,
then we call ${\cal B}$ a {\em zero forcing set};
otherwise we call ${\cal B}$ a {\em failed zero forcing set}
and ${\cal W}$ a {\em zero blocking set}.
The {\em zero forcing number} $Z(G)$ of $G$ is the minimum size of a zero forcing set.
The {\em failed zero forcing number} $F(G)$ of $G$ is the maximum size of a failed zero forcing set.
The {\em zero blocking number} $B(G)$ of $G$ is the minimum size of a zero blocking set.
As $F(G)+B(G)=|V(G)|$ for every graph $G$, studying $F(G)$ and $B(G)$ are equivalent.
Notice that if a zero blocking set of size $B(G)$,
then there exists no black vertex adjacent to exactly one white vertex.

The zero forcing process was introduced in \cite{2008AIM}
to study minimum rank problems in linear algebra, and independently in \cite{2007bg}
to explore quantum systems memory transfers in quantum physics.
The computation of the zero forcing number of a graph is NP-hard \cite{2008a}.
Considerable effort has been made to find exact values and bounds of
this number for specific classes of graphs,
and to investigate a variety of related concepts arising in zero forcing processes.

The concept of failed zero forcing number was first introduced in \cite{2015fjs}.
The computation of the failed zero forcing number is NP-hard \cite{2017s}.
Results for exact values and bounds of this number were also established in
{\cite{2024agm,2020apnaA,2020apnaB,2016ajps,2020bccknt,2024c,2021grtn,
2024grtn,2020ksv,2023ktn,2025LLc,2023su}.
Our main concern is the upper bounds for the zero blocking numbers of the
grid graph given by Beaudouin-Lafon et al.~\cite{2020bccknt}.
In particular, we prove that their bound is in fact the exact value.

For integers $m,n\ge 1$, the {\em grid graph} $G_{m,n}=P_m\Box P_n$
is the graph with vertex set $\{(i,j):1\le i\le n,\, 1\le j\le m\}$,
where $(x,y)$ is adjacent to $(x',y')$ if $|x-x'|+|y-y'|=1$.
The {\em $x$-th column} is the set $\{(x,j):1\le j\le m\}$
and the {\em $y$-th row} is the set $\{(i,y):1\le i\le n\}$.
Remark that in the books/papers of the graph theory, the vertices of
$G\Box H$ are usually denoted by $(u, v)$ with $u\in V(G)$ and $v\in V(H)$, and
vertices of $V(G)$ are drawn vertically and $V(H)$ horizontally. This is different
from the notation for the $x$-$y$ plane $\mathbb{R}^2$
in the coordinate geometry. In this
paper, we use the notion of geometry as we need to describe lattice points, line
segments and rays etc. In this way, a vertex $(x,y)$ in $G_{m,n}$ means $x\in V(P_n)$
and $y\in V(P_m)$.

Beaudouin-Lafon et al.~\cite{2020bccknt} gave an upper bound of $B(G_{m,n})$ in two cases.

\begin{theorem}[{\cite{2020bccknt}}]\label{upper}
For the grid graph $G_{m,n}$ with $n\ge m\ge 2$,
we have $B(G_{m,n})\le m(q+1)-\lfloor r/2\rfloor$
for $\lceil\frac{n-m}{m+1}\rceil \le \lfloor\frac{n-m}{m-1}\rfloor$
and $B(G_{m,n})\le n+m-\lfloor\frac{n-m}{m+1}\rfloor-2$
for $\lceil\frac{n-m}{m+1}\rceil > \lfloor\frac{n-m}{m-1}\rfloor$,
where $q=\lceil\frac{n-m}{m+1}\rceil$ and $r=(m+1)q-(n-m)$.
\end{theorem}

The purpose of this paper is to
prove that the upper bound is in fact the exact value of $B(G_{m,n})$.

\section{Terminology}\label{sec2}

Recall that elements in $\mathbb{R}^2$ are called points,
elements in $\mathbb{Z}^2$ are lattice points
and elements in $V(G_{m,n})$ are vertices.
Suppose $A,B,A_1,A_2,\ldots,A_k$ are points in $\mathbb{R}^2$.
\begin{itemize}
\item The coordinate representation of a point $A$ is $(x_A,y_A)$.

\item The {\em line segment} between $A$ and $B$ is
$$\overline{AB} = \{((1-t)x_A+tx_B,(1-t)y_A+ty_B): 0\le t\le 1\}.$$

\item Denote $\overline{A_1 A_2 \ldots A_k}=\bigcup_{i=1}^{k-1} \overline{A_i A_{i+1}}$.

\item The {\em northeast ray starting at $A$} is $\NEarrow{A}=\{(x_A+t,x_A+t): t\ge 0\}$.
Similarly define
$\SEarrow{A}$, $\NWarrow{A}$, $\SWarrow{A}$,
$\Eastarrow{A}$, $\Westarrow{A}$, $\Northarrow{A}$ and $\Southarrow{A}$
for the directions {\em southeast}, {\em northwest},
{\em southwest}, {\em east}, {\em west}, {\em north} and {\em south}, respectively.

\item The {\em northeast point from $A$ at distance $d$} is $\NEArrow{A}{d}=(x_A+d,y_A+d)$.
Similarly define
$\SEArrow{A}{d}$, $\NWArrow{A}{d}$, $\SWArrow{A}{d}$,
$\EastArrow{A}{d}$, $\WestArrow{A}{d}$, $\NorthArrow{A}{d}$ and $\SouthArrow{A}{d}$.

\item For $B\in\NEarrow{A}$, denote
$\overline{AB^+}=\overline{A\NEArrow{B}{1/2}}$ and
$\overline{AB^-}=\overline{A\SWArrow{B}{1/2}}$.

\item Denote $\North(A)$ the set
$\{\NWArrow{A}{1},\NorthArrow{A}{1},\NEArrow{A}{1},\NorthArrow{A}{2}\}$.
Similarly define $\South(A)$, $\East(A)$ and $\West(A)$.
\end{itemize}

For $A\in \mathbb{R}^2$ and $S\subseteq \mathbb{R}^{2}$,
we say that $A$ is {\em below} $S$
if there exists $B\in S$ such that $x_A= x_B$ and $y_A\le y_B$,
and {\em strictly below} $S$ if $x_A=x_B$ and $y_A<y_B$.
Similarly define {\em above} and {\em left to} and {\em right to}.

First two useful lemmas.

\begin{lemma}[\cite{2020bccknt}] \label{lem2}
If $A=(x_A,y_A)$ is a vertex in a minimum zero blocking set ${\cal W}$ of $G_{m,n}$
with $y_A<m$, then ${\cal W}$ intersects $\North(A)$.
Similar properties hold for $y_A>1$ with $\South(A)$,
$x_A<n$ with $\East(A)$, and $x_A>1$ with $\West(A)$.
\end{lemma}

\begin{lemma}[\cite{2020bccknt}] \label{lem3}
Every minimum zero blocking set of $G_{m,n}$ intersects
the first column, the last column and any pair of consecutive columns, respectively.
Similar property holds for rows.
\end{lemma}

\section{Basic properties}\label{sec3}

Now we choose a minimum zero blocking set $\cal W$ of the grid graph $G_{m,n}$.
Let $X=(1,1)$, $Y=(n,1)$, $Z=(1,m)$ and $W=(n,m)$
be the four {\em corner} vertices of $G_{m,n}$,
Suppose $A_1,A_2,\ldots,A_k$ are the white vertices in $\overline{XY}$ from left to right,
that is, $\{A_1,A_2,\ldots,A_k\}=\overline{XY}\cap{\cal W}$ with
$1\le x_{A_1} <x_{A_2} <\cdots<x_{A_k} \le n$ and $y_{A_i}=1$ for $i=1,2,\ldots,k$.
Name the following points
\begin{itemize}
\item $B_0 = \Northarrow{X} \cap \NWarrow{A_1} = (1,x_{A_1})$,

\item $B_i = \NEarrow{A_i} \cap \NWarrow{A_{i+1}}
           = (\frac{x_{A_{i+1}}+x_{A_i}}{2},\frac{x_{A_{i+1}}-x_{A_i}}{2}+1)$
      for $1\le i\le k-1$,

\item $B_k = \NEarrow{A_k} \cap \Northarrow{Y} = (n,n+1-x_{A_k})$.
\end{itemize}
Denote $S_{\overline{XY}}=\overline{B_0A_1B_1A_2\ldots B_k}$,
see Figure \ref{fig1} for an example.
Similarly define $S_{\overline{ZW}}$, $S_{\overline{XZ}}$ and $S_{\overline{YW}}$.

\begin{figure}[htb]
\centering
\begin{tikzpicture}[scale=3/4]
\draw[very thick,gray!50] (1,1) grid (11,6);
\draw (1,1) node [below left] {$X$} (11,1) node [below right] {$Y$}
      (1,6) node [above left] {$Z$} (11,6) node [above right] {$W$};
\draw[very thick] (1,3)--(3,1)--(5.5,3.5)--(8,1)--(9,2)--(10,1)--(11,2);
\draw[very thick,fill=white]
(3,1) circle (.15) node [below] {$A_1$}
(8,1) circle (.15) node [below] {$A_2$}
(10,1) circle (.15) node [below] {$A_3$};
\draw
(1,3) node [left] {$B_0$}
(5.5,3.5) node [above] {$B_1$}
(9,2) node [above] {$B_2$}
(11,2) node [right] {$B_3$};
\end{tikzpicture}
\caption{An example of $S_{\overline{XY}}$ in $G_{6,11}$.} \label{fig1}
\end{figure}
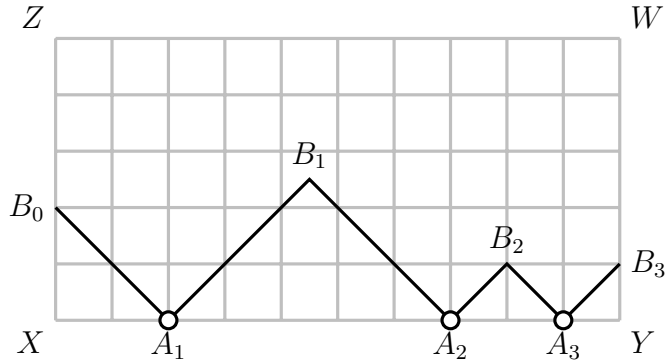

First, we consider some basic properties that are useful in this paper.

\begin{proposition}\label{prop4}
The vertices strictly below $S_{\overline{XY}}$ are black,
and the vertices in $S_{\overline{XY}} \setminus\{B_1,B_2,\ldots,B_{k-1}\}$ are white.
\end{proposition}
\begin{proof}
Suppose to the contrary that there exists a white vertex strictly below $S_{\overline{XY}}$.
Choose such a white vertex $A$ with $y_A$ minimum.
Then $y_A\ge 2$, and all vertices in $\South(A)$ are black
and strictly below $S_{\overline{XY}}$,
which contradicts Lemma \ref{lem2}.
So the vertices strictly below $S_{\overline{XY}}$ are black.
Next, suppose to the contrary that there exists a black vertex in
$S_{\overline{XY}} \setminus \{B_1,B_2,\ldots,B_{k-1}\}$.
Choose such a black vertex $B$ with $y_B$ minimum.
Then $y_B\ge 2$, and either $\SWArrow{B}{1}$ or $\SEArrow{B}{1}$
is a white vertex in $S_{\overline{XY}}\setminus\{B_1,B_2,\ldots,B_{k-1}\}$.
If $C=\SWArrow{B}{1}$ (respectively, $C=\SEArrow{B}{1}$) is white, then
$x_C\le n-1$ (respectively, $x_C\ge 2$) and all vertices in
$\East(C)\setminus\{B\}$ (respectively, $\West(C)\setminus\{B\}$)
are strictly below $S_{\overline{XY}}$ and hence black.
Together with $B$, the vertices in $\East(C)$ (respectively, $\West(C)$) are black,
which contradicts Lemma \ref{lem2}.
Therefore, the vertices in $S_{\overline{XY}}\setminus\{B_1,B_2,\ldots,B_{k-1}\}$ are white.
\end{proof}

\begin{proposition}\label{prop5}
In $S_{\overline{XY}}$,
we have $y_{B_i}\le m+1$ for $1\le i\le k-1$ and $y_{B_i}\le m$ for $i\in\{0,k\}$.
Consequently, the lattice points in
$S_{\overline{XY}} \setminus \{B_1,B_2,\ldots,B_{k-1}\}$ are vertices of $G_{m,n}$.
\end{proposition}
\begin{proof}
If there exists $y_{B_i}>m+1$ for some $1\le i\le k-1$,
then $x_{A_{i+1}}-x_{A_i}= 2(y_{B_i}-1)>2m$.
By Proposition \ref{prop4},
the $(x_{A_i}+m)$-th and
the $(x_{A_i}+m+1)$-th columns are a pair of consecutive columns
which ${\cal W}$ does not intersect,
contradicting Lemma \ref{lem3}.
If there exists $y_{B_i}>m$ for some $i\in\{0,k\}$,
then by Proposition \ref{prop4},
${\cal W}$ does not intersect the first column or the last ($n$-th) column,
contradicting Lemma \ref{lem3}.
\end{proof}

For $S\subseteq\mathbb{R}^2$,
let $[S]=|S\cap\mathbb{Z}^2|$ denote the number of lattice points in $S$.
For a white vertex $A$ on $\overline{XY}$ and $B\in \NEarrow{A}$ with $x_B\le n$,
we denote $S_{\overline{XY}}(\overline{AB})=\{C\in S_{\overline{XY}}: x_A\le x_C\le x_B\}$.
Then $[S_{\overline{XY}}(\overline{AB})]=[\overline{AB}]$.
For each $B_i\in S_{\overline{XY}}(\overline{AB})$
which is a lattice point but not a white vertex,
we have $1\le i\le k-1$ and $\SEArrow{B_i}{1}\in S_{\overline{XY}}$ is a white vertex.
If $x_{B_i}+1>x_B$, then let $B_i'=\SEArrow{B_i}{1}$.
Otherwise, $x_{B_i}+1\le x_{B}$.
Let $d_i=\max\{d: \NEArrow{(\SEArrow{B_i}{1})}{d}\in
S_{\overline{XY}}(\overline{AB})\cap {\cal W}\}$
and $C_i=\NEArrow{(\SEArrow{B_i}{1})}{d_i}$.
If $y_{C_i}\le m-1$, then ${\cal W}$ intersects $\North(C_i)$ by Lemma \ref{lem2} and
let $B_i'$ be a white vertex in $\North(C_i)$.
Otherwise, $y_{C_i}=m$,
and we say $B_i$ is a {\em hole} of $S_{\overline{XY}}(\overline{AB})$.
Let $W_{\overline{XY}}(\overline{AB})$
be the set of white vertices in $S_{\overline{XY}}(\overline{AB})$
together with those $B_i'$,
and let $H_{\overline{XY}}(\overline{AB})$ be the set of holes.
See Figure \ref{fig2} for an example.
Similarly we may define
$W_{\overline{XY}}(\overline{AB})$ and $H_{\overline{XY}}(\overline{AB})$
for $B\in\NWarrow{A}$.

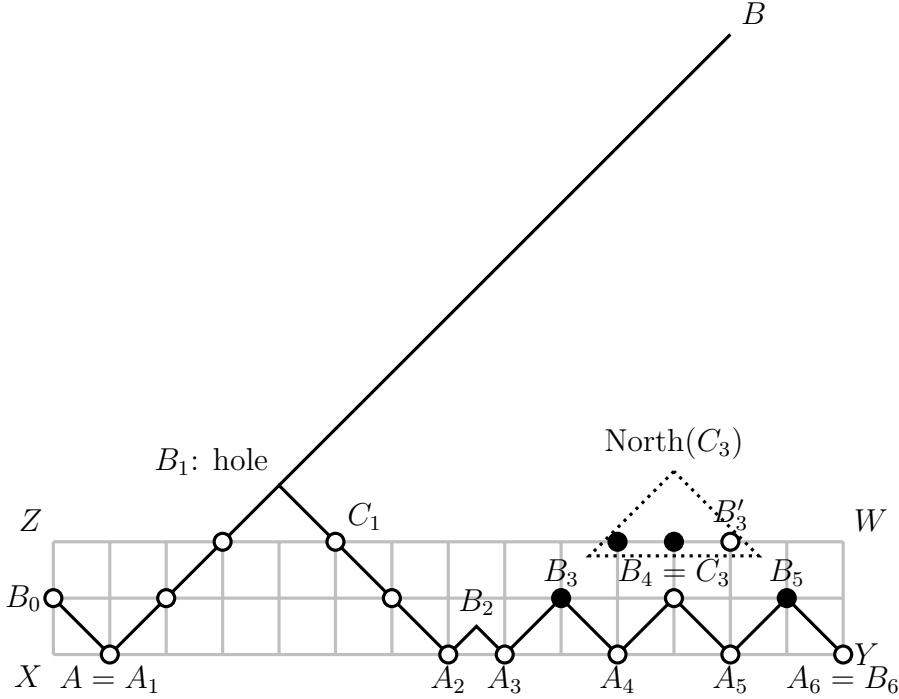
\begin{figure}[htb]
\centering
\begin{tikzpicture}[scale=3/4]
\draw[very thick,gray!50] (1,1) grid (15,3);
\draw (1,1) node [below left] {$X$} (15,1) node [right] {$Y$}
      (1,3) node [above left] {$Z$} (15,3) node [above right] {$W$};
\draw[very thick] (1,2)--(2,1)--(5,4)--(8,1)--(8.5,1.5)--(9,1)--(10,2)--(11,1)--(12,2)
--(13,1)--(14,2)--(15,1) (5,4)--(13,12);
\draw[very thick,fill=white]
(2,1) circle (.15) node [below] {$A=A_1$}
(8,1) circle (.15) node [below] {$A_2$}
(9,1) circle (.15) node [below] {$A_3$}
(11,1) circle (.15) node [below] {$A_4$}
(13,1) circle (.15) node [below] {$A_5$}
(15,1) circle (.15) node [below] {$A_6=B_6$}
(1,2) circle (.15) node [left] {$B_0$}
(3,2) circle (.15) (4,3) circle (.15)
(6,3) circle (.15) node [above right] {$C_1$}
(7,2) circle (.15)
(12,2) circle (.15) node [above] {$B_4=C_3$}
(13,3) circle (.15) node [above] {$B_3'$};
\draw[very thick,fill]
(10,2) circle (.15) node [above] {$B_3$}
(14,2) circle (.15)node [above] {$B_5$}
(12,3)circle (.15) (11,3)circle (.15) ;
\draw
(5,4) node [above left] {$B_1$: hole}
(8.5,1.5) node [above] {$B_2$}
(13,12) node [above right] {$B$};
\draw[very thick,dotted]
 (12,4.25)--(10.5,2.75)--(13.5,2.75)--(12,4.25) node[above] {$\North(C_3)$} ;
\end{tikzpicture}
\caption{An example for $B_i'$ and a hole.} \label{fig2}
\end{figure}

\begin{proposition}\label{prop6}
If $A$ is a white vertex on $\overline{XY}$ and $B\in\NEarrow{A}$ with $x_B\le n$,
then the following properties of $W_{\overline{XY}}(\overline{AB})$ hold.
\begin{enumerate}
\item[(1)]
$|W_{\overline{XY}}(\overline{AB})|=[\overline{AB}]-|H_{\overline{XY}}(\overline{AB})|$.

\item[(2)]
The vertices in $W_{\overline{XY}}(\overline{AB})$
are below $\overline{AB}\cup\SEarrow{B}$,
more precisely, $\overline{AB\SEArrow{B}{1}}$.

\item[(3)]
If $B_i$ is a hole with $x_{B_i}=2$,
then the vertices in $W_{\overline{XY}}(\overline{AB})$ are below
$\NWarrow{(\NWArrow{B_i}{1})} \cup \NEarrow{(\SEArrow{B_i}{1})}$.
\end{enumerate}
\end{proposition}
\begin{proof}
(1)
Suppose $B_i,B_j\in S_{\overline{XY}}(\overline{AB})$ are lattice points
that are neither white vertices nor holes.
It is sufficient to show that $B_i'$ is not in
$S_{\overline{XY}}(\overline{AB})\cap\mathcal{W}$,
and $B_i'\ne B_j'$ for $i\ne j$.
If $x_{B_i}+1>x_B$, then $B_i'=\SEArrow{B_i}{1}$ and
so $x_{B_i'}= x_{B_i}+1>x_B$.
Hence $B_i'$ is not in $S_{\overline{XY}}(\overline{AB})\cap\mathcal{W}$.
If $x_{B_i}+1\le x_B$,
then $B_i'\in \North(C_i)$ is either $\NEArrow{C_i}{1}$
or strictly above $\NEarrow{(\SEArrow{B_i}{1})}$,
we have $B_i'$ is not in $S_{\overline{XY}}(\overline{AB})\cap \mathcal{W}$
by the maximality of $d_i$.
If $i<j$, then $x_{B_i}+1\le x_{B_j}\le x_B$.
Since $x_{B_i'}\le x_{C_i}+1\le x_{B_j}\le x_{B_j'}-1$,
we have $B_i'\ne B_j'$.
Therefore,
$|W_{\overline{XY}}(\overline{AB})|
=[S_{\overline{XY}}(\overline{AB})]-|H_{\overline{XY}}(\overline{AB})|
=[\overline{AB}]-|H_{\overline{XY}}(\overline{AB})|$.

(2)
Let $P=\SEArrow{A}{1}$, $Q=\SEArrow{B}{1}$
and $R=\SWArrow{Q}{1}=\SouthArrow{B}{2}$.
It is clear that the vertices in $S_{\overline{XY}}(\overline{AB})$
are below $\overline{AB}$, and so are below $\overline{AB} \cup\SEarrow{B}$.
Suppose $B_i\in S_{\overline{XY}}(\overline{AB})$ is a lattice point
which is neither a white vertex nor a hole.
Then $B_i$ is below $\overline{AB}$ which implies $\SEArrow{B_i}{1}$ is below $\overline{PQ}$.
If $x_{B_i}+1>x_B$, then $B_i'=\SEArrow{B_i}{1}$
is below $\overline{PQ}$ and so is below $\overline{AB}\cup\SEarrow{B}$.
Otherwise, $x_{B_i}+1\le x_B$.
then $C_i$ is left to $\Southarrow{B}$ and below $\overline{PR}$.
Therefore, $B_{i}'\in \North(C_i)$ is below $\overline{ABQ}$
and so is below $\overline{AB}\cup\SEarrow{B}$.
See Figure \ref{fig3} for a demonstration.

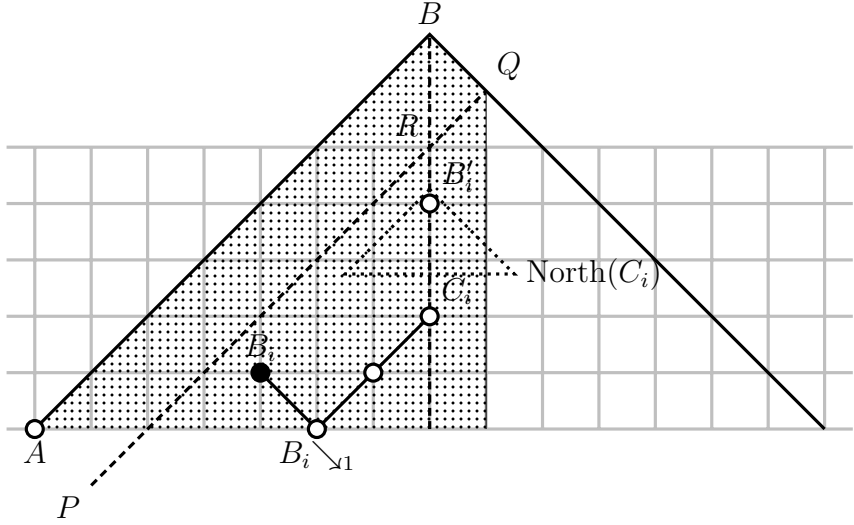
\begin{figure}[htb]
\centering
\begin{tikzpicture}[scale=3/4]
\draw[very thick,gray!50] (.5,1) grid (15.5,6);
\draw[pattern=dots] (1,1)--(8,8)--(9,7)--(9,1);
\draw[very thick] (1,1)--(8,8)--(15,1) (5,2)--(6,1)--(8,3);
\draw[very thick,densely dashed ] (8,1)--(8,8) (2,0)--(9,7);
\draw[very thick,fill=white]
(1,1) circle (.15) node [below] {$A$}
(6,1) circle (.15) node [below] {$\SEArrow{B_i}{1}$}
(7,2) circle (.15)
(8,3) circle (.15) node [above right] {$C_i$}
(8,5) circle (.15) node [above right] {$B_i'$};
\draw[very thick,fill]
(5,2) circle (.15) node [above] {$B_i$};
\draw
(8,8) node [above ] {$B$}
(2,0) node [below left] {$P$}
(9,7) node [above right] {$Q$}
(8,6) node [above left] {$R$};
\draw[very thick,dotted]
(8,5.25)--(6.5,3.75)--(9.5,3.75)--(8,5.25) (9.5,3.75) node [right] {$\North(C_i)$};
\end{tikzpicture}
\caption{Demonstration for the proof of case (2), Proposition \ref{prop6}.
The vertices in $W_{\overline{XY}}(\overline{AB})$ are in the shaded area.} \label{fig3}
\end{figure}

(3)
It is clear that the vertices in $S_{\overline{XY}}(\overline{AB})\cap \mathcal{W}$
are below $\NWarrow{(\NWArrow{B_i}{1})} \cup \NEarrow{(\SEArrow{B_i}{1})}$.
Suppose $B_j\in S_{\overline{XY}}(\overline{AB})$ is a lattice point
that is neither a white vertex nor a hole.
If $x_{B_j}>x_{B_i}$, then $B_j'$ is below $\NEarrow{(\SEArrow{B_i}{1})}$.
Otherwise, $x_{B_j}<x_{B_i}$ and so $x_{B_j}+1\le x_{B_i}\le x_B$.
Then $C_j\in S_{\overline{XY}}$ and $x_A<x_{C_j}<x_{B_i}$,
so $C_j$ is below $\NWarrow{(\SWArrow{B_i}{1})}$.
Hence $B_j'\in \North(C_j)$ is below $\NWarrow{B_i}$.
Since $B_j'\ne B_i$, we have $B_j'$ is below $\NWarrow{(\NWArrow{B_i}{1})}$.
See Figure \ref{fig4} for a demonstration.
\end{proof}

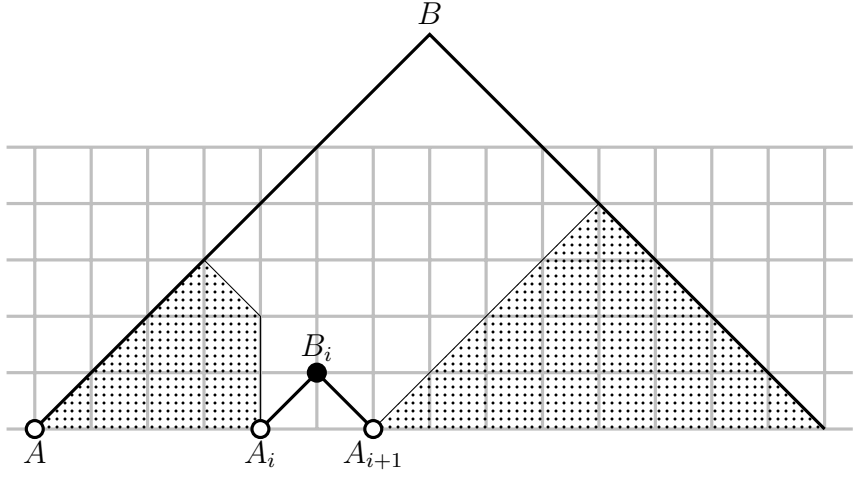
\begin{figure}[htb]
\centering
\begin{tikzpicture}[scale=3/4]
\draw[very thick,gray!50] (.5,1) grid (15.5,6);
\draw[pattern=dots] (1,1)--(4,4)--(5,3)--(5,1) (7,1)--(11,5)--(15,1);
\draw[very thick] (1,1)--(8,8)--(15,1) (5,1)--(6,2)--(7,1) ;
\draw[very thick,fill=white]
(1,1) circle (.15) node [below] {$A$}
(5,1) circle (.15) node [below] {$A_i$}
(7,1) circle (.15) node [below] {$A_{i+1}$};
\draw[very thick,fill]
(6,2) circle (.15) node [above] {$B_i$};
\draw
(8,8) node [above] {$B$};
\end{tikzpicture}
\caption{Demonstration for the proof of case (3), Proposition \ref{prop6}.
The vertices in $W_{\overline{XY}}(\overline{AB})$ are in the shaded area.} \label{fig4}
\end{figure}

\begin{proposition}\label{prop7}
If $A$ is a white vertex in $\overline{XY}$ and $B\in\NEarrow{A}$ with $x_B\le n$,
then the following properties of $H_{\overline{XY}}(\overline{AB})$ hold.
\begin{enumerate}
\item[(1)]
If $B_i$ is a hole, then $y_{B_i}\in\{2,m+1\}$ and $x_{B_i}\le n-m$.

\item[(2)]
If $B_i$ and $B_j$ are distinct holes, then $|x_{B_i}-x_{B_j}|\ge m+1$.

\item[(3)]
There is at most one hole $B_i$ with $x_A<x_{B_i}<x_A+m$,
and such a hole must satisfy $y_{B_i}=2$.

\item[(4)]
If $x_B<x_A+m+1$, then there is no hole.
\end{enumerate}
\end{proposition}
\begin{proof}
(1)
Since $2\le y_{B_i}\le m+1$,
if $3\le  y_{B_i}\le m$, then we have $d=0$ and so $y_{C_i}=y_{B_i}-1\le m-1<m$,
a contradiction.
Hence $y_{B_i}\in\{2,m+1\}$.
If $y_{B_i}=m+1$, then $n\ge x_{A_{i+1}}=x_{B_i}+m$ which implies $x_{B_i}\le n-m$.
Otherwise, $y_{B_i}=2$ and $y_{C_i}=m$.
Hence $n\ge x_{C_i}\ge x_{B_i}+1+(m-1)$ which implies $x_{B_i}\le n-m$.

(2)
We may assume that $i<j$.
If $y_{B_i}=m+1$, then $x_{B_j}\ge x_{A_{i+1}}+1=x_{B_i}+m+1$.
Otherwise, $y_{B_i}=2$ and $y_{C_i}=m$.
Hence $x_{B_j}\ge x_{C_i}+1=x_{B_i}+m+1$.

(3)
The uniqueness follows from (2).
If $y_{B_i}=m+1$, then $x_{B_i}=x_{A_i}+m\ge x_A+m$, a contradiction.
Hence $y_{B_i}=2$.

(4)
Suppose to the contrary that there exists a hole $B_i$.
If $y_{B_i}=m+1$, then $d=0$ and
so $x_B\ge x_{C_i}=x_{B_i}+1=x_{A_i}+m+1\ge x_A+m+1$,
a contradiction.
Otherwise, $y_{B_i}=2$ and $y_{C_i}=m$,
then $x_B\ge x_{C_i}
= x_{A_{i+1}}+(m-1)= (x_{A_i}+2)+(m-1)\ge x_A+m+1$, a contradiction.
\end{proof}

\section{Exact value of $B(G_{m,n})$}\label{sec2}

Finally, we shall prove that the upper bound by Beaudouin-Lafon et al.~\cite{2020bccknt}
is in fact the exact value of $B(G_{m,n})$ with $n\ge m\ge 2$.
First, a useful lemma.

\begin{lemma}\label{lem8}
For $n\ge m\ge 2$, suppose $n-m=q(m+1)-r$,
where $q$ and $r$ are integers with $0\le r\le m$.
If there exist integers $a,b,c\ge 0$ such that $n-m=a(m-1)+bm+c(m+1)$,
then $c\le q-\lceil r/2\rceil$ and so $r\le 2q$.
\end{lemma}
\begin{proof}
Let $s=a+b+c$.
Since $a+b\ge \frac{2a+b}{2}=\frac{s(m+1)-(n-m)}{2}$,
we have $c\le s-\frac{s(m+1)-(n-m)}{2}$,
which is decreasing with respect to $s$.
Since $s(m+1)-(n-m)=2a+b\ge 0$,
we have $s\ge\lceil\frac{n-m}{m+1}\rceil=q$
and so $c\le q-\frac{q(m+1)-(n-m)}{2}=q-{r}/{2}$.
Since $c$ and $q$ are integers,
we have $c\le q-\lceil r/2\rceil$.
\end{proof}

\begin{theorem}\label{thm9}
For the grid graph $G_{m,n}$ with $n\ge m\ge 2$,
if $n-m=q(m+1)-r$, where $q$ and $r$ are integers with $0\le r\le m$,
then $B(G_{m,n})=n-q+\lceil r/2\rceil$ for $r\le 2q$
and $B(G_{m,n})=n-q+m-1$ for $r>2q$.
\end{theorem}
\begin{proof}
First, we prove that the upper bound on $B(G_{m,n})$
in Theorem \ref{upper} is the same as the exact value in this theorem.
Notice that $q=\lceil\frac{n-m}{m+1}\rceil$ is a non-negative integer.
Hence the condition $\lceil\frac{n-m}{m+1}\rceil\le\lfloor\frac{n-m}{m-1}\rfloor$ is
the same as $q\le\frac{n-m}{m-1}$, or equivalently $r=q(m+1)-(n-m)\le 2q$.
Also $n-q+\lceil r/2\rceil = m+qm-r+\lceil r/2\rceil = m(q+1)-\lfloor r/2\rfloor$.
On the other hand, the condition
$\lceil\frac{n-m}{m+1}\rceil > \lfloor\frac{n-m}{m-1}\rfloor$ is equivalent to $r>2q$.
In this case, $r>0$ and so $q=\lceil\frac{n-m}{m+1}\rceil=\lfloor\frac{n-m}{m+1}\rfloor+1$.
Also $n-q+m-1=n+m-\lfloor\frac{n-m}{m+1}\rfloor-2$.
Therefore, it is sufficient to prove the lower bound.
We consider the following cases.

(1)
Consider the case when there exists a white vertex
$A\in(\overline{XY}\cup\overline{ZW})\setminus\{X,Y,Z,W\}$,
say $A\in\overline{XY}\setminus\{X,Y\}$.

(1.1)
Consider the subcase when there exists a white vertex $B\in\overline{ZW}$
strictly above $\NWarrow{A}\cup\NEarrow{A}$, that is, $|x_A-x_B|<m-1$.
In this subcase, let
$C=\NWarrow{A}\cap\SWarrow{B}$, $D=\NEarrow{A}\cap\SEarrow{B}$,
$E=\SWarrow{B}\cap\Southarrow{Z}$ and $F=\SEarrow{B}\cap\Southarrow{W}$.
See Figure \ref{fig5} for a demonstration.

\begin{figure}[htb]
\centering
\begin{tikzpicture}[scale=3/4]
\draw[very thick,gray!50] (1,1) grid (11,6);
\draw (1,1) node [below left] {$X$} (11,1) node [below right] {$Y$}
      (1,6) node [above left] {$Z$} (11,6) node [above right] {$W$};
\draw[pattern=dots] (5,6)--(11,0)--(11,6) (7,1)--(8,2)--(9,1);
\draw[very thick] (3.5,4.5)--(7,1)--(8.5,2.5) (1,2)--(5,6)--(11,0) ;
\draw[very thick,dotted] (11,1) --(11,-.5);
\draw[very thick,fill=white]
(7,1) circle (.15) node [below] {$A$}
(5,6) circle (.15) node [above] {$B$};
\draw
(3.5,4.5) node [above left] {$C$}
(8.5,2.5) node [left] {$D$}
(8,2) node [left] {$D^-$}
(1,2) node [below left] {$E$}
(11,0) node [below right] {$F$};
\end{tikzpicture}
\caption{Vertices/points $A, B, C, D, E, F$ in case (1.1), Theorem \ref{thm9}.
The vertices in $W_{\overline{XY}}(\overline{AD^*})$ and in $W_{\overline{XY}}(\overline{BF^*})$
are in the disjoint shaded area when $\overline{AD^*}=\overline{AD^-}$.} \label{fig5}
\end{figure}
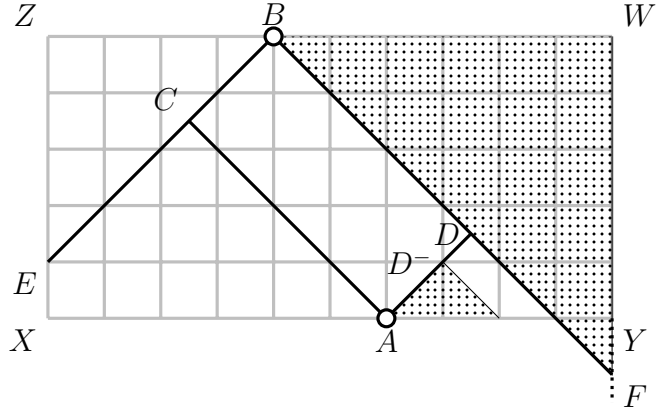

(1.1.1)
If $x_{C}\ge 1$ and $x_{D}\le n$, then by Proposition \ref{prop7} (3),
there exists at most one hole $P$ of $S_{\overline{ZW}}(\overline{BF})$
with $x_{B}<x_{P}<x_{D}+1\le x_{B}+m$, and such $P$ must satisfy $y_{P}=m-1$.
Let $\overline{AD^*}=\overline{AD^+}$ if $P$ exists
and $\overline{AD^*}=\overline{AD^-}$ for otherwise.
Then by Proposition \ref{prop6} (2) and (3),
$W_{\overline{XY}}(\overline{AD^*})$ and $W_{\overline{ZW}}(\overline{BF})$ are disjoint.
Similar property holds for $\overline{AC^*}$.
By Proposition \ref{prop7} (4), there exists no hole of
$S_{\overline{XY}}(\overline{AC^*})$ and $S_{\overline{XY}}(\overline{AD^*})$.
Since $[\overline{AD^+}]=[\overline{AD^-}]+1$, we have
$[\overline{AD^*}]=[\overline{AD^-}]+|\{P\in H_{\overline{ZW}}(\overline{BF}):x_B<x_P<x_D+1\}|$.
Now consider the set
$W_{\overline{XY}}(\overline{AD^*})\cup W_{\overline{XY}}(\overline{AC^*})
\cup W_{\overline{ZW}}(\overline{BF})\cup W_{\overline{ZW}}(\overline{BE})$.
By Proposition \ref{prop6} (1), there exist at least
$[\overline{AD^-}]+[\overline{AC^-}]+[\overline{BF}]+[\overline{BE}]-[\{A,B\}]-
|\{P\in H_{\overline{ZW}}(\overline{BF}): x_P\ge x_D+1\}|-
|\{Q\in H_{\overline{ZW}}(\overline{BE}): x_Q\le x_C-1\}|$ white vertices.
Since $A$ and $B$ are lattice points, we have
$[\overline{AD^-}]+[\overline{AC^-}]=[\overline{BC^-}]+[\overline{AC^-}]
=[\overline{BCA}]-[\{C\}]\ge m-1$
and $[\overline{BF}]+[\overline{BE}]=[\overline{FBE}]+[\{B\}]=n+1$.
For a hole $P\in H_{\overline{ZW}}(\overline{BF})$ with $x_P\ge x_D+1\ge x_C+m\ge 1+m$
and a hole $Q\in H_{\overline{ZW}}(\overline{BE})$ with $x_Q\le x_C-1\le x_D-m\le n-m$,
we have $|x_P-x_Q|\ge (x_D+1)-(x_C-1)=m+1$.
Hence, by Proposition \ref{prop7} (1) and (2),
there exist at most
$\lfloor \frac{(n-m)-(1+m)}{m+1}\rfloor +1=\lfloor\frac{n-m}{m+1}\rfloor$
such holes in total.
Therefore, there exist at least $(m-1)+(n+1)-2-\lfloor\frac{n-m}{m+1}\rfloor
=n+m-\lfloor\frac{n-m}{m+1}\rfloor-2$ white vertices.
See Figure \ref{fig5} and \ref{fig6} for a demonstration.

\begin{figure}[htb]
\centering
\begin{tikzpicture}[scale=3/4]
\draw[very thick,gray!50] (1,1) grid (11,6);
\draw (1,1) node [below left] {$X$} (11,1) node [below right] {$Y$}
      (1,6) node [above left] {$Z$} (11,6) node [above right] {$W$};
\draw[pattern=dots] (7,1)--(9,3)--(11,1) (7,6)--(11,2)--(11,6);
\draw[very thick] (3.5,4.5)--(7,1)--(8.5,2.5) (1,2)--(5,6)--(11,0);
\draw[very thick,dotted] (11,1) --(11,-.5);
\draw[very thick,fill=white]
(7,1) circle (.15) node [below] {$A$}
(5,6) circle (.15) node [above] {$B$};
\draw[very thick,fill]
(6,5) circle (.15) node [above] {$P$};
\draw
(3.5,4.5) node [above left] {$C$}
(8.5,2.5) node [left] {$D$}
(9,3) node [above] {$D^+$}
(1,2) node [below left] {$E$}
(11,0) node [below right] {$F$};
\end{tikzpicture}
\begin{tikzpicture}[scale=3/4]
\draw[very thick,gray!50] (1,1) grid (11,6);
\draw (1,1) node [below left] {$X$} (11,1) node [below right] {$Y$}
      (1,6) node [above left] {$Z$} (11,6) node [above right] {$W$};
\draw[pattern= dots] (7,1)--(9,3)--(11,1) (11,5)--(10,6)--(11,6) (7.5,3.5)--(8,4)--(8,6)--(5,6);
\draw[very thick] (3.5,4.5)--(7,1)--(8.5,2.5) (1,2)--(5,6)--(11,0);
\draw[very thick,dotted] (11,1)--(11,-.5);
\draw[very thick,fill=white]
(7,1) circle (.15) node [below] {$A$}
(5,6) circle (.15) node [above] {$B$};
\draw[very thick,fill]
(9,5) circle (.15) node [above] {$P$};
\draw
(3.5,4.5) node [above left] {$C$}
(8.5,2.5) node [left] {$D$}
(9,3) node [above] {$D^+$}
(1,2) node [below left] {$E$}
(11,0) node [below right] {$F$};
\end{tikzpicture}
\caption{Demonstration for the proof of case (1.1.1), Theorem \ref{thm9}.
The vertices in $W_{\overline{XY}}(\overline{AD^*})$
and in $W_{\overline{XY}}(\overline{BF^*})$
are in the disjoint shaded area when $\overline{AD^*}=\overline{AD^+}$.} \label{fig6}
\end{figure}
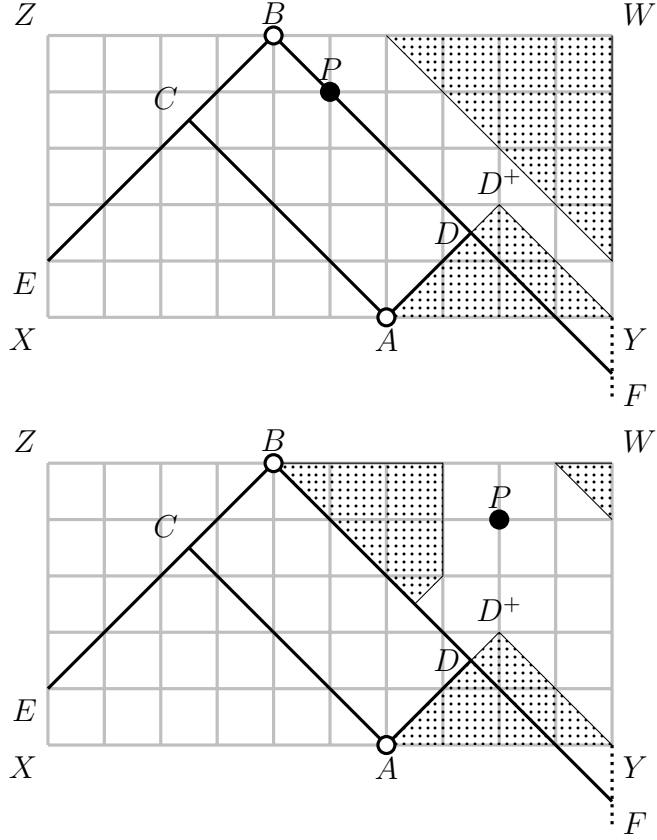

(1.1.2)
If $x_{C}<1$ or $x_{D}>n$, say $x_{C}<1$,
then $x_{D}=x_{C}+m-1<m\le n$.
Similar to Case (1.1.1),
there exists at most one hole $P$ of $S_{\overline{ZW}}(\overline{BF})$
with $x_{B}<x_{P}<m<x_{B}+m$,
and we define $\overline{AD^*}$ by the existence of such $P$.
Now we consider two configurations as follows.

(1.1.2.1)
If $X$ is black,
since $x_B\leq m$, every $A_i$ left to $A$ has the same property as $A$,
then we choose $A=A_1$ so that $B_0=\overline{AC} \cap \overline{XZ}$ is white.
Let $G=\NEarrow{B_0} \cap \SEarrow{E}$.
Now we consider the set
$W_{\overline{XY}}(\overline{AD^*})\cup W_{\overline{XY}}(\overline{AB_0^-})
\cup W_{\overline{XZ}}(\overline{B_0G^-})
\cup W_{\overline{ZW}}(\overline{BF})\cup W_{\overline{ZW}}(\overline{BE})$.
By Proposition \ref{prop6} (1), there exist at least
$[\overline{AD^-}]+[\overline{AB_0^-}]+[\overline{B_0G^-}]+[\overline{BF}]+[\overline{BE}]
-[\{A,B\}]-|\{P\in H_{\overline{ZW}}(\overline{BF}): x_P\ge m\}|$
white vertices.
Since $A,B$ and $B_0$ are lattice points,
we have $[\overline{AD^-}]+[\overline{AB_0^-}]+[\overline{B_0G^-}]
=[\overline{BC^-}]+[\overline{AC^-}]=[\overline{BCA}]-[\{C\}]\ge m-1$
and $[\overline{BF}]+[\overline{BE}]=[\overline{FBE}]+[\{B\}]=n+1$.
By Proposition \ref{prop7} (1) and (2), we have
$|\{P\in H_{\overline{ZW}}(\overline{BF}):x_P\ge m\}|
\le\lfloor\frac{(n-m)-m}{m+1}\rfloor+1=\lfloor\frac{n-m+1}{m+1}\rfloor$.
Therefore, there exist at least $(m-1)+(n+1)-2-\lfloor\frac{n-m+1}{m+1}\rfloor
=n+m-\lfloor\frac{n-m+1}{m+1}\rfloor-2$ white vertices.
See Figure \ref{fig7} for a demonstration.

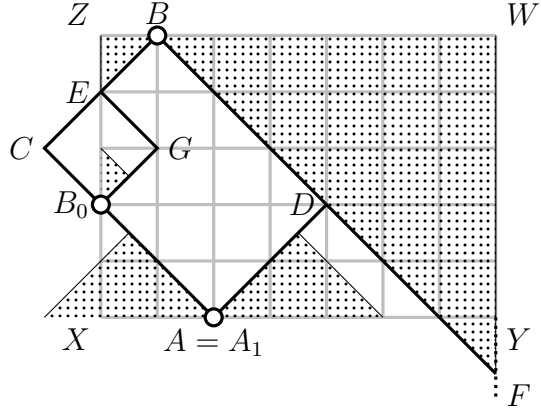
\begin{figure}[htb]
\centering
\begin{tikzpicture}[scale=3/4]
\draw[very thick,gray!50] (4,1) grid (11,6);
\draw (4,1) node [below left] {$X$} (11,1) node [below right] {$Y$}
      (4,6) node [above left] {$Z$} (11,6) node [above right] {$W$};
\draw[pattern=dots] (11,6)--(11,0)--(5,6)--(4,5)--(4,6) (9,1)--(7.5,2.5)--(6,1)--(4.5,2.5)--(3,1) (4,3)--(4.5,3.5)--(4,4);
\draw[very thick] (11,0)--(5,6)--(3,4)--(6,1)--(8,3) (4,3)--(5,4)--(4,5);
\draw[very thick,dotted] (11,1) --(11,-.5);
\draw[very thick,fill=white]
(6,1) circle (.15) node [below] {$A=A_1$}
(5,6) circle (.15) node [above] {$B$}
(4,3) circle (.15) node [left] {$B_0$};
\draw
(8,3) node [left] {$D$}
(3,4) node [left] {$C$}
(4,5) node [left] {$E$}
(11,0) node [below right] {$F$}
(5,4)  node [right] {$G$};
\end{tikzpicture}
\caption{Demonstration for the proof of case (1.1.2.1), Theorem \ref{thm9}.} \label{fig7}
\end{figure}

(1.1.2.2)
If $X$ is white, then let
$\NEarrow{X}$ intersect $\overline{AC}$, $\SEarrow{E}$ and $\overline{BD}$ at $H$, $I$ and $J$, respectively.
Now we consider the set
$W_{\overline{XY}}(\overline{AD^*})\cup W_{\overline{XY}}(\overline{AH^-})
\cup W_{\overline{XZ}}(\overline{XI^-})
\cup W_{\overline{ZW}}(\overline{BF})\cup W_{\overline{ZW}}(\overline{BE})$.
By Proposition \ref{prop6} (1), there exist at least
$[\overline{AD^-}]+[\overline{AH^-}]+[\overline{XI^-}]+[\overline{BF}]+[\overline{BE}]-[\{A,B\}]-
|\{P\in H_{\overline{ZW}}(\overline{BF}): x_P\ge m\}|$
white vertices.
Since $A,X,E$ and $B$ are lattice points,
we have $[\overline{AH^-}] \ge [\overline{DJ^-}]$ and
$[\overline{XI^-}]=[\overline{EI^-}]=[\overline{BJ^-}]$.
Then $[\overline{AD^-}]+[\overline{AH^-}]+[\overline{XI^-}] \ge [\overline{BDA}] - [\{J\}]\ge m-1$
and $[\overline{BF}]+[\overline{BE}]=[\overline{FBE}]+[\{B\}]=n+1$.
By Proposition \ref{prop7} (1) and (2), we have
$|\{P\in H_{\overline{ZW}}(\overline{BF}): x_P\ge m\}|\le\frac{(n-m)-m}{m+1}+1
=\frac{n-m+1}{m+1}$.
Therefore, there exist at least $(m-1)+(n+1)-2-\lfloor\frac{n-m+1}{m+1}\rfloor
=n+m-\lfloor\frac{n-m+1}{m+1}\rfloor-2$ white vertices.
See Figure \ref{fig8} for a demonstration.

\begin{figure}[htb]
\centering
\begin{tikzpicture}[scale=3/4]
\draw[very thick,gray!50] (4,1) grid (11,6);
\draw (4,1) node [below] {$X=A_1$} (11,1) node [below right] {$Y$}
      (4,6) node [above left] {$Z$} (11,6) node [above right] {$W$};
\draw[pattern=dots] (11,6)--(11,0)--(5,6)--(4,5)--(4,6) (9,1)--(7.5,2.5)--(6,1)--(5.5,1.5)--(5,1) (4,1)--(5.5,2.5)--(4,4);
\draw[very thick] (11,0)--(5,6)--(3,4)--(6,1)--(8,3) (4,1)--(7,4) (6,3)--(4,5);
\draw[very thick,dotted] (11,1)--(11,-.5);
\draw[very thick,fill=white]
(6,1) circle (.15) node [below] {$A$}
(5,6) circle (.15) node [above] {$B$}
(4,1) circle (.15);
\draw
(3,4) node [left] {$C$}
(8,3) node [left] {$D$}
(4,5) node [left] {$E$}
(11,0) node [below right] {$F$}
(5,2)  node [right] {$H$}
(6,3)  node [right] {$I$}
(7,4)  node [below] {$J$};
\end{tikzpicture}
\caption{Demonstration for the proof of case (1.1.2.2), Theorem \ref{thm9}.} \label{fig8}
\end{figure}
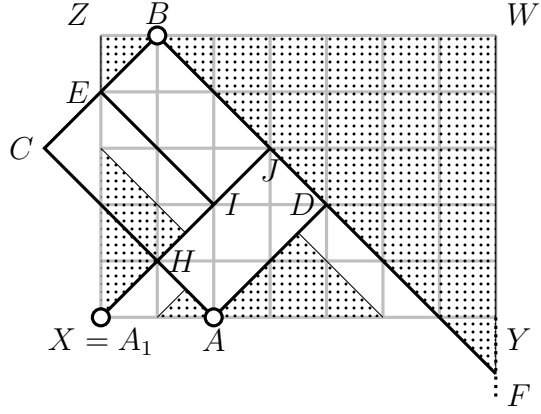

(1.2)
Consider the subcase that
for each $A_i\in\overline{XY}\setminus\{X,Y\}$,
there exists no white vertex in $\overline{ZW}$
strictly above $\NWarrow{A_i}\cup \NEarrow{A_i}$, and
for all white vertices $A\in\overline{ZW}\setminus\{Z,W\}$,
there exists no white vertex in $\overline{XY}$
strictly below $\SWarrow{A}\cup \SEarrow{A}$.
Let $K=\SWarrow{A}\cap \Westarrow{Y}$ and $L=\SEarrow{A}\cap \Eastarrow{X}$.
Since $B_0\in\overline{XZ}$ is strictly left to $\overline{AL}$,
we have $A_1\in \SEarrow{B_0}$ is strictly left to $L$, and hence left to $K$.
Similarly, $A_k$ is right to $L$.
Hence, $A$ is below $\overline{A_iB_iA_{i+1}}$ for some $1\leq i\leq k-1$,
and by Proposition \ref{prop4},
the vertices strictly below $\overline{A_iB_iA_{i+1}}$ are black,
which implies $K=A_i$ or $L=A_{i+1}$.
If $K=A_i$, then we have $x_A-x_{A_i}=m-1$
and, since $m =y_A \le y_{B_i} \le m+1$ by Proposition \ref{prop5},
$x_{A_{i+1}}-x_{A}
=(x_{A_{i+1}}-x_{B_i})+(x_{B_i}-x_{A})
=(y_{B_i}-y_{A_{i+1}})+(y_{B_i}-y_{A})
=2y_{B_i}-1-m\in\{m-1,m,m+1\}$;
in symmetry, $x_A-x_{A_i}\in\{m-1,m,m+1\}$ and $x_{A_{i+1}}-x_A=m-1$ when $L=A_{k+1}$.
Then, we repeat the process by taking place of $A$ with $A_i$ and $A_{i+1}$,
and it stops at corner vertices
and produces integers $a,b,c\ge 0$ with $(a+1)(m-1)+bm+c(m+1)=n-1$,
that is, $a(m-1)+bm+c(m+1)=n-m$,
which implies that there exist at least $n-c$ white vertices.

(2)
If there exists no white vertex other than $X,Y,Z$ and $W$,
then by Lemma \ref{lem3}, we have $X$ or $Y$ is white, say $X$ is white.

(2.1)
If $n=m$, then $W_{\overline{XY}}(\overline{XW})$ has at least $n$ white vertices.

(2.2)
If $n>m$, then $Y$ is white. Otherwise
$y_{B_{1}}=n> m$, a contradiction.
Similarly, $Z$ and $W$ are white.
Meanwhile, we have $n\in\{m+1,m+2\}$. Otherwise, by Proposition \ref{prop4},
the vertices strictly below $S_{\overline{XY}}$
and the vertices strictly above $S_{\overline{ZW}}$
produce a pair of consecutive columns of black vertices,
which contradicts to Lemma \ref{lem3}.
Also, $S_{\overline{XY}}$ and $S_{\overline{ZW}}$ contain at least $2m$ white vertices.

Having all cases analyzed, we now evaluate the lower bound.
\begin{itemize}
\item
In case (1.1.1), we have
$|\mathcal{W}|\ge n+m-\lfloor\frac{n-m}{m+1}\rfloor-2$.
If $r=0$, then $r\le 2q$ and
$n+m-\lfloor\frac{n-m}{m+1}\rfloor-2=n-q+m-2\ge n-q+\lceil r/2\rceil$.
If $r>0$, then
$n+m-\lfloor\frac{n-m}{m+1}\rfloor-2=n-(q-1)+m-2=n-q+m-1
\ge n-q+\lceil r/2\rceil$.
\item
In both cases (1.1.2.1) and (1.1.2.2), we have $m\ge 3$ and
$|\mathcal{W}|\ge n+m-\lfloor\frac{n-m+1}{m+1}\rfloor-2$.
If $r\le 1$, then $r\le 2q$ and
$n+m-\lfloor\frac{n-m+1}{m+1}\rfloor-2=n-q+m-2
\ge n-q+\lceil r/2\rceil$.
If $r>1$, then
$n+m-\lfloor\frac{n-m+1}{m+1}\rfloor-2=n-(q-1)+m-2=n-q+m-1
\ge n-q+\lceil r/2\rceil$.
\item
In case (1.2), by Lemma \ref{lem8}, we have $r\le 2q$
and $|\mathcal{W}|\ge n-c\ge n-q+\lceil r/2\rceil$.
\item
In case (2.1), we have $r=q=0\le 2q$ and $|{\cal W}|\ge n=n-q+\lceil r/2\rceil$.
\item
In case (2.2), we have $n\in\{m+1,m+2\}$ and $|\mathcal{W}|\ge 2m$.
So $q=1$ and $2m\ge n-q+m-1\ge n-q+\lceil r/2\rceil$.
\end{itemize}

Therefore, we have
$|{\cal W}|\ge n-q+\lceil r/2\rceil$ for $r\le 2q$
and $|{\cal W}|\ge n-q+m-1$ for $r> 2q$,
and the proof for the theorem then completes.
\end{proof}


\begin{thebibliography}{99}
\bibitem{2008a}
A. Aazami,
Hardness results and approximation algorithms for some problems on graphs,
Ph.D. thesis, University of Waterloo, 2008.

\bibitem{2024agm}
F. Afzali, A.H. Ghodrati and H.R. Maimani,
Failed zero forcing numbers of Kneser graphs, Johnson graphs, and hypercubes,
J. Applied Math. Comput. 70 (2024) 2665--2675.

\bibitem{2008AIM}
AIM Minimum Rank --- Special Work Group,
Zero forcing sets and the minimum rank of graphs,
Linear Algebra Appl. 428 (2008) 1628--1648.

\bibitem{2020apnaA}
P.A.B. Pelayo and M.N.M. Abara,
Maximal failed zero forcing sets for products of two graphs (2022).
arXiv: 2202.04997

\bibitem{2020apnaB}
P.A.B. Pelayo and M.N.M. Abara,
Minimum rank and failed zero forcing number of graphs (2022).
arXiv:2202.04993

\bibitem{2016ajps}
T.  Ansill, B. Jacob, J. Penzellna and D. Saavedra,
Failed skew zero forcing on a graph,
Linear Algebra Appl. 509 (2016) 40--63.

\bibitem{2020bccknt}
M. Beaudouin-Lafon, M. Crawford, S. Chen, N. Karst, L. Nielsen and D.S. Troxell,
On the zero blocking number of rectangular, cylindrical, and M\"obius grids,
Discrete Appl. Math. 282 (2020) 35--47.

\bibitem{2007bg}
D. Burgarth and V. Giovannetti,
Full control by locally induced relaxation,
Phys. Rev. Lett. 99 (2007), 100501.

\bibitem{2024c}
Y.-C. Chou,
Bounds on the Zero Blocking Number of Graphs,
Master thesis, Dept. Applied Math., National Yang Ming Chiao Tung Univ., 2024.

\bibitem{2015fjs}
K. Fetcie, B. Jacob and D. Saavedra,
The failed zero forcing number of a graph,
Involve 8 (2015) 99--117.

\bibitem{2021grtn}
L. Gomez, K. Rubi, J. Terrazas, R. Florez and D.A. Narayan,
All graphs with a failed zero forcing number of two,
Symmetry 13 (2021), 2221.

\bibitem{2024grtn}
L. Gomez, K. Rubi, J. Terrazas and D.A. Narayan,
Failed zero forcing numbers of trees and circulant graphs,
Theory Appl. Graphs 11 (2024), 5.

\bibitem{2020ksv}
N. Karst, X. Shen and M. Vu,
Blocking zero forcing processes in Cartesian products of graphs,
Discrete Appl. Math. 285 (2020) 380--396.

\bibitem{2023ktn}
C. Kaudan, R. Taylor and D.A. Narayan,
An inverse approach for finding graphs with a failed zero forcing number of $k$,
Mathematics 11 (2023), 4068.

\bibitem{2025LLc}
H.-Y. Lin, W.-H. Lin and G.J. Chang,
The zero blocking numbers of generalized Kneser graphs and generalized Johnson graphs,
submitted.

\bibitem{2017s}
Y. Shitov,
On the complexity of failed zero forcing,
Theor. Comput. Sci. 660 (2017) 102--104.

\bibitem{2023su}
N. Swanson and E. Ufferman,
A lower bound on the failed zero-forcing number of a graph,
Involve 16 (2023) 493--504.

\end{thebibliography}
\end{document}